\documentclass{amsart} % add \maketitle before abstract
\usepackage{a4wide, enumerate}
\usepackage[%dvipdfm,
backref=page,%colorlinks,
bookmarksopen=true]{hyperref}
\usepackage{amssymb,amsmath,amsthm,amscd,graphicx,color, mathrsfs}
\usepackage[all]{xy}
\usepackage{enumerate}

\usepackage[latin1]{inputenc}
\usepackage[T1]{fontenc}       % accents dans le DVI

\newtheorem{thme}{Theorem}
\newtheorem{thm}{Theorem}[section]

\newtheorem{prop}[thm]{Proposition}
\newtheorem{lemma}[thm]{Lemma}
\newtheorem*{conjecture}{Conjecture}
\theoremstyle{definition}

\theoremstyle{remark}

\newcommand{\NN}{\mathbb{N}}
\newcommand{\PP}{\mathbb{P}}

\newcommand{\RR}{\mathbb{R}}

\newcommand{\ZZ}{\mathbb{Z}}
\newcommand{\norma}{\mathscr{N}}
\newcommand{\centra}{\mathscr{Z}}

\newcommand{\CCC}{\mathcal{C}}

\newcommand{\PPP}{\mathcal{P}}

\newcommand{\inv}{^{-1}}

\newcommand{\la}{\langle}
\newcommand{\ra}{\rangle}

\newcommand{\cat}{CAT($0$) }

\def\co{\colon\thinspace}

\DeclareMathOperator{\Aut}{Aut}

\DeclareMathOperator{\CAT}{CAT(0)}
\DeclareMathOperator{\Ch}{Ch}
\DeclareMathOperator{\charact}{char}

\DeclareMathOperator{\id}{id}

\DeclareMathOperator{\Min}{Min}

\DeclareMathOperator{\op}{op}
\DeclareMathOperator{\proj}{proj}

\DeclareMathOperator{\Stab}{Stab}

\begin{document}

\title{Can an anisotropic reductive group admit a Tits system?}

\author{Pierre-Emmanuel {Caprace}*}
\address{Universit\'e catholique de Louvain\\ Chemin du Cyclotron 2\\ 1348 Louvain-la-Neuve\\ Belgium}
\email{pierre-emmanuel.caprace@uclouvain.be}
\thanks{*Supported by the Fund for Scientific Research--F.N.R.S., Belgium}

\author{Timoth\'ee \textsc{Marquis}}
\address{Universit\'e libre de Bruxelles\\ Boulevard du Triomphe\\ 1050 Bruxelles\\ Belgium}
\email{tmarquis@ulb.ac.be}

\date{August 2009}

% AMS classification numbers
 %(Compact groups)

\begin{abstract}
Seeking for a converse to a well-known theorem by Borel--Tits, we address the question whether the group of rational points $G(k)$ of an anisotropic reductive $k$-group may admit a split spherical BN-pair. We show that if $k$ is a perfect field or a local field, then such a BN-pair must be virtually trivial. We also consider arbitrary compact groups and show that the only abstract BN-pairs they can admit are spherical, and even virtually trivial provided they are split.
\end{abstract}

\maketitle

%\tableofcontents

\section{Introduction}
In a seminal paper~\cite{BT}, Armand Borel and Jacques Tits
established --- amongst other things --- that the group $G(k)$ of
$k$-rational points of a (connected) reductive  linear algebraic
$k$-group $G$ always possesses a canonical BN-pair, where $k$ is an
arbitrary ground field. More precisely, they showed that if $P$ is a
minimal parabolic $k$-subgroup of $G$, and if $N$ is the normalizer
in $G$ of some maximal $k$-split torus contained in $P$, then
$(P(k),N(k))$ is a BN-pair for $G(k)$. This result constitutes a
cornerstone in understanding the abstract group structure of the
group of $k$-rational points $G(k)$. As an application, it yields
for example the celebrated simplicity result of Tits~\cite{Ti64}.
Of course, the aforementioned BN-pair is trivial when  $G$ is
\textbf{anisotropic} over $k$. (Abusing slightly the standard
conventions, we shall say that $G$ is anisotropic if it has no
proper $k$-parabolic subgroup, \emph{i.e.} if $P = G$. As is
well-known, this definition coincides with the standard one in case
$G$ is semi-simple (see \cite[11.21]{Borel})). In fact, the abstract group
structure of $G(k)$ remains intriguing and mysterious  to a large
extent in the anisotropic case. In this context, we propose the
following.

\begin{conjecture}[Converse to Borel--Tits]
Let $G$ be a reductive algebraic $k$-group which is anisotropic over $k$. Then every  split spherical BN-pair for $G(k)$ is trivial.
\end{conjecture}

Recall that a BN-pair $(B,N)$ for a group $G$ is called {\bf
spherical} if the associated Weyl group $W:=N/T$ is finite, where
$T:=B\cap N$. It is said to be {\bf split} if it is  saturated
(\emph{i.e.} $T=\bigcap_{w\in W}{wBw\inv}$), and if there exists a
nilpotent normal subgroup $U\triangleleft B$ such that $B\cong
U\rtimes T$. Note that if $(B,N)$ is irreducible of rank at least
$2$, one can show that $U$ is automatically nilpotent
(see~\cite{Tent}).  The BN-pair for $G(k)$ described above is always
split in the above sense (\cite[14.19]{Borel}).

Besides the natural search for a converse to Borel--Tits, a
motivation to consider the above conjecture is provided by the
recent work of Peter Abramenko and Ken Brown~\cite{AB07}, who
constructed Weyl transitive actions on trees for certain anisotropic
groups over global function fields. We refer to~\cite[Ch.~6]{ABrown}
for more details on the relations and distinctions between BN-pairs,
strong transitivity and Weyl transitivity.

Our first contribution concerns the special case when the ground
field $k$ is a local field. The $k$-anisotropy of $G$ is then
equivalent to the compactness of $G(k)$ (see~\cite{Prasad}). In
fact, our first step will be to establish the following two results,
which concern arbitrary compact topological groups (not necessarily
associated with algebraic groups).

\begin{thme}\label{1}
Let $G$ be a compact group. Then every BN-pair for $G$ is spherical.
\end{thme}

\begin{thme}\label{2}
Let $G$ be a compact group possessing a split spherical BN-pair $(B,N)$. Then, the associated building is finite. In other words, $[G:B]<\infty$.
\end{thme}

We emphasize that the BN-pairs appearing in these statements are
\emph{abstract}: The corresponding subgroups $B$ and $N$ are
\emph{not} supposed to be closed in $G$. Specializing to anisotropic
groups over local fields, we deduce the following immediate
corollary.

\begin{thme}\label{3}
Let $k$ be a local field and $G$ be a connected semi-simple algebraic $k$-group which is anisotropic over $k$. Then:
\begin{enumerate}[(1)]
\item
Every BN-pair for $G(k)$ is spherical.
\item
Every split spherical BN-pair $(B,N)$ for $G(k)$ is `virtually trivial', in the sense that $B$ has finite index in  $G(k)$.
\end{enumerate}
\end{thme}

Finally, we consider the case of perfect ground fields.

\begin{thme}\label{4}
Let $k$ be a perfect field and $G$ be a reductive algebraic $k$-group which is anisotropic over $k$.  Then every split spherical BN-pair for $G(k)$ is virtually trivial.
\end{thme}
Notice that  Theorems~\ref{3} and~\ref{4} are logically independent, since there exist local fields which are not perfect and vice-versa.

It would be very interesting to sharpen the conclusion of
Theorems~\ref{3} and~\ref{4}, that is, to show that the BN-pair must
be trivial, and not only virtually trivial. However, we expect this
to be quite difficult, since it is closely related  to a conjecture
due to Andrei Rapinchuk and Gopal Prasad
(see~\cite{Rapinchukprasad}), which may be stated as follows:
``\emph{Let $G$ be a reductive $k$-group which is anisotropic over
$k$. Then, every finite quotient of $G(k)$ is solvable.}'' As of
today, this conjecture was confirmed only when $G$ is the multiplicative
group of a finite dimensional division algebra (see
\cite{Rapinchuk}). We now sketch informally how these two problems
are related.

On one side, if $G(k)$ possesses a BN-pair with finite associated building $\Delta$, and if $K:=\ker(G(k)\curvearrowright \Delta)$ is the kernel of the corresponding action, then $G(k)/K$ is a finite group whose action on $\Delta$ is faithful, and thus $G(k)/K$ possesses a faithful BN-pair. But these groups have been classified: they are simple Chevalley groups, and in particular are not solvable  (up to two exceptions). Thus, if the BN-pair for $G(k)$ were nontrivial, there would exist (modulo the two exceptions) a non-solvable finite quotient of $G(k)$.

Conversely, suppose that $G(k)$ possesses a nontrivial and
non-solvable finite quotient $F':=G(k)/K$. Let $R\lneq F'$ be the
solvable radical of $F'$, that is, its largest solvable normal
subgroup. Going to the quotient $F:=F'/R$, we thus know that $G(k)$
surjects onto a nontrivial finite group with trivial solvable
radical (namely, $F$). Let now $M$ be a minimal normal subgroup of
$F$. Then $M$ is a direct product of non-Abelian simple groups which
are pairwise isomorphic, say $M\cong S_1\times\dots\times S_k$ with
$S_i\cong S$ for all $i\in\{1,\dots,k\}$. By the classification of
finite simple groups, $S$ is very likely to be a Chevalley group.
Such a group possesses a root datum, and thus also a nontrivial
BN-pair whose associated (finite) building is in bijection with
$S/B$. Repeating this construction for each $S_i$, we then get a
finite building $\Delta=\Delta_1\times\dots\times\Delta_k$ on which
$M=S_1\times\dots\times S_k$ acts strongly transitively. Finally,
the action of $\Aut(M)$ on the set of $p$-Sylow subgroups of $M$
(where $p=\charact k$) induces an action of $\Aut(M)$ on $\Delta$
making the diagram
$$\minCDarrowwidth40pt\begin{CD}
F @>\alpha>>  \Aut(M)\\
@A\iota AA @VVV \\
M @>>\textrm{strongly tr.}> \Aut(\Delta)
\end{CD}$$
commutative, where $\alpha(f)$ denotes the conjugation by $f$ for all $f\in F$. In particular, we get a strongly transitive action of $F$, and thus also of $G(k)$, on the finite building $\Delta$. This yields a nontrivial and virtually trivial BN-pair for $G(k)$.

\subsection*{General conventions} All algebraic groups considered here are supposed to be affine, all topological groups are assumed Hausdorff and all BN-pairs have finite rank.

\subsection*{Acknowledgement} We are very grateful to the anonymous referee for his/her useful detailed comments.

\section{Proof of Theorem 1}

\subsection{Heuristic sketch}
Let $G$ be a compact group and let $(B,N)$ be a BN-pair for $G$.
Also, let $\Delta$ be the associated building. We consider the Davis
realization of $\Delta$, noted $|\Delta|_{\CAT}$ in this paper, and
which is a complete $\CAT$ space, as well as a simplicial complex,
on which $G$ acts by simplicial isometries. The key step in the
proof of Theorem 1 is then to establish that this action is elliptic
(Theorem~\ref{thm compact elliptic} below). To do so, we use a
result of Martin Bridson stating that such an action is always
semi-simple, and we then argue by contradiction, assuming that $G$
possesses an element with no fixed point. Such an element would then
generate a subgroup $Q$ of $G$ which  acts by translations on
$|\Delta|_{\CAT}$. Moreover, the structure of simplicial complex of
$|\Delta|_{\CAT}$ implies that the set of translation lengths of the
elements of $Q$ is discrete at $0$. The contradiction now comes from
divisibility properties of compact and procyclic groups, which we
apply to $Q$.

\subsection{Procyclic groups}
Let $G$ be a profinite group. Recall that $G$ is said to be {\bf procyclic} if there exists a $g\in G$ such that the subgroup generated by $g$ is dense in $G$, that is, $G=\overline{\la g\ra}$. Moreover $G$ is said to be {\bf pro-$p$} for some prime $p$ if every finite Hausdorff quotient of $G$ is a  $p$-group.

The following basic properties of procyclic groups can be found
in~\cite[2.7]{Procyclic}. The symbol $\PP$ denotes the set of all
primes.

\begin{prop}\label{prop structure procyclic}
Let $G$ be a procyclic group. Then,
\begin{enumerate}[(i)]
\item
$G$ is the direct product $G=\prod_{p\in\PP}{G_p}$ of its $p$-Sylow subgroups, and each $G_p$ is a pro-$p$ procyclic group.
\item
$G$ is, in a unique way, a quotient of $\hat{\ZZ}:=\prod_{p\in\PP}{\ZZ_p}$. If $G$ is pro-$p$ for some $p\in\PP$, then it is a quotient of $\ZZ_p$.
\end{enumerate}
\end{prop}

\subsection{Divisible groups}
Recall that an element $g\in G$ is said to be {\bf $n$-divisible} for some $n\in\NN$ if there exists an $h\in G$ such that $h^n=g$. We say that $g$ is {\bf divisible} if it is $n$-divisible for each $n\geq 1$. The group $G$ is called {\bf $n$-divisible} (respectively {\bf divisible}) when all its elements are.

Now, every prime $q$ different from $p$ is invertible in $\ZZ_p$
since its $p$-adic valuation is zero. Hence, the additive group
$\ZZ_p$ is $q$-divisible for each $q\in\PP\setminus\{p\}$. In
particular, Proposition~\ref{prop structure procyclic} shows that if
a procyclic group $G$ has trivial $q$-Sylow subgroups, then $G$ is
$q$-divisible.

We conclude this paragraph by stating the following characterization
of divisibility for compact groups (see~\cite[Corollaire
2]{connexedivisible}).
\begin{prop}\label{prop compact iff divisible}
Let $G$ be a compact topological group. Then, $G$ is divisible if and only if it is connected.
\end{prop}

\subsection{Semi-simple actions on \cat spaces}
Let $G$ be a group acting on a metric space $(X,d)$. For every $g\in G$, we define the {\bf translation length} of $g$ by $|g|:=\inf\{d(x,g\cdot x) \ | \ x\in X\}\in [0,\infty)$ and the {\bf minimal set} of $g$ by $\Min(g):=\{x\in X \ | \ d(x,g\cdot x)=|g|\}$. An element $g\in G$ is said to be {\bf semi-simple} when $\Min(g)$ is nonempty. In that case, we say that $g$ is {\bf elliptic} if it fixes some point, that is, if $|g|=0$; otherwise, if $|g|>0$, we call $g$ {\bf hyperbolic}.

A {\bf geodesic line} (respectively, {\bf geodesic segment}) in $X$ is an isometry $f\co\RR\to X$ (respectively, $f\co [0;1]\to X$); by abuse of language, we will identify $f$ with its image in $X$.

The following lemma follows from Proposition 2.4
in~\cite{BridsonHaefliger}.
\begin{lemma}\label{lemma existence projection}
Let $(X,d)$ be a complete $\CAT$ metric space, and let $C$ be a closed convex nonempty subset of $X$. Then:
\begin{enumerate}[(i)]
\item
For every $x\in X$, there is a unique $y\in C$ such that $d(x,y)=d(x,C)$, where $d(x,C):=\inf_{z\in C}{d(x,z)}$. We call $y$ the {\bf projection} of $x$ on $C$ and we write $y=\proj_Cx$.
\item
For all $x_1,x_2\in X$, we have $d(\proj_Cx_1,\proj_Cx_2)\leq d(x_1,x_2)$.
\end{enumerate}
\end{lemma}

Suppose now that $(X,d)$ is a cell complex. We then say that $G$
acts by {\bf cellular isometries} on $X$ if it preserves the metric,
as well as the cell decomposition of $X$.

The following result is due to Martin Bridson~\cite{Bridson}.
\begin{prop}\label{thm Bridson}
Let $X$ be a locally Euclidean $\CAT$ cell complex with finitely many isometry types of cells, and  $G$ be a group acting on $X$ by cellular isometries. Then every element of $G$ is semi-simple. Moreover, $\inf\{|g|\neq 0 \ | \ g\in G\}>0$.
\end{prop}

We now establish the following result, which is the key ingredient for the proof of Theorem 1:

\begin{thm}\label{thm compact elliptic}
Let $X$ be a locally Euclidean $\CAT$ cell complex with finitely many isometry types of cells, and  $G$ be a compact group acting on $X$ by cellular isometries (not necessarily continuously). Then every element of $G$ is elliptic.
\end{thm}
\begin{proof}
Suppose for a contradiction there exists a $g\in G$ without fixed point. Proposition~\ref{thm Bridson} then implies that $g$ is hyperbolic. Let $Q=\overline{\la g \ra}$ be the closure of the subgroup generated by $g$ in $G$. So, $Q$ is compact.

\medskip \noindent
\underline{Claim 1}: \emph{ $Q$ is Abelian.}

\medskip \noindent This is clear since it
contains a dense Abelian (in fact cyclic) subgroup.

\medskip \noindent
\underline{Claim 2}: \emph{For every $h\in Q$, the minimal set $\Min(h)$ is a closed convex subset of $X$ which is stabilized by $Q$.}

\medskip \noindent
This follows from~\cite[Proposition II.6.2]{BridsonHaefliger}.

\medskip \noindent
\underline{Claim 3}: \emph{For every $h\in Q$ and every nonempty
closed convex subset $C$ of $X$ stabilized by $Q$, the set
$C\cap\Min(h)$ is nonempty.}

\medskip \noindent
Note first that $\Min(h)$ is nonempty
by Proposition~\ref{thm Bridson}. Let $x\in \Min(h)$ and consider
the projections $y:=\proj_Cx$ and $z:=\proj_Chx$ provided by
Lemma~\ref{lemma existence projection}. Since $hC=C$, we then obtain
$$d(x,y)=\inf_{c\in C}d(x,c)=\inf_{c\in C}d(hx,hc)=\inf_{c\in C}d(hx,c)=d(hx,z).$$
Hence $d(hx,hy)=d(x,y)=d(hx,z)$, and so $z=hy=\proj_Chx$ by uniqueness
of projections. Since in addition $d(y,z)\leq d(x,hx)=|h|$ by
Lemma~\ref{lemma existence projection}, we finally get
$d(y,hy)=|h|$ and therefore $y\in C\cap\Min(h)$.

\medskip \noindent
\underline{Claim 4}: \emph{For all $h_1,h_2\in Q$, the set
$\Min(h_1)\cap\Min(h_2)$ is nonempty.}

\medskip \noindent
As $\Min(h_1)$ and $\Min(h_2)$ are nonempty by Proposition~\ref{thm
Bridson}, the claim follows from Claims~2 and~3.

\medskip \noindent
\underline{Claim 5}: \emph{Let $h\in Q$ and let $C$ be a nonempty
closed convex subset of $X$ stabilized by $Q$. We may thus consider
the action of $h$ on $C$. Denote by $|h|_C$ the translation length
of $h$ for this action. Then, $h$ is semi-simple in $C$ and
$|h|=|h|_C$.}

\medskip \noindent
Claim 3 yields that if $x\in \Min(h)$, then $y:=\proj_Cx\in\Min(h)$.
Since $\Min(h)$ is nonempty by Proposition~\ref{thm Bridson}, the
claim follows.

\medskip \noindent
\underline{Claim 6}: \emph{For every $h\in Q$ and $n\geq 1$, we have
$|h^n|=n|h|$.}

\medskip \noindent
By Claim 4, we may choose an $x\in \Min(h)\cap\Min(h^n)$. Note that
$h$ is elliptic (respectively hyperbolic) if and only if $h^n$ is so (see
\cite[II.6.7 and II.6.8]{BridsonHaefliger}). In particular, if $h$
is hyperbolic, then $x$ belongs to some $h$-axis, which is also an
$h^n$-axis. In any case, we obtain  $d(x,h^nx)=nd(x,hx)$, whence
$|h^n|=d(x,h^nx)=nd(x,hx)=n|h|$.

\medskip \noindent
\underline{Claim 7}: \emph{Every divisible element of $Q$ is
elliptic.}

\medskip \noindent
Let $h\in Q$ be divisible and suppose for a contradiction it is not
elliptic. Then $h$ is hyperbolic by Proposition~\ref{thm Bridson}.
For each natural number $n\geq 1$, choose an $h_n\in Q$ such that
$h_n^n=h$. In particular, all $h_n$ are hyperbolic. Moreover,
$|h_n^n|=n|h_n|$ by Claim 6. Therefore, we obtain a sequence $(h_n)$
of elements of $Q$ such that $|h_n|=|h|/n>0$, contradicting the
second part of Proposition~\ref{thm Bridson}.

We now establish the desired contradiction to the hyperbolicity of
$g$. First note that the component group $P:=Q/Q^0$ of $Q$ is a
profinite group. In fact, it is even procyclic, since the subgroup
generated by the projection of $g$ in $P$ is dense in $P$, the
natural mapping $\pi\co Q\to Q/Q^0$ being continuous. In particular,
it follows from Proposition~\ref{prop structure procyclic} that $P$
is the product of its $p$-Sylow subgroups $P_p$. Moreover, each
$P_p$ is a pro-$p$ group and is therefore $q$-divisible for every
$q\in\PP\setminus\{p\}$. For each $p\in\PP$, let $Q_p$ be the
subgroup of $Q$ which is the pre-image of $P_p$ under $\pi$.

\medskip \noindent
\underline{Claim 8}: \emph{If $h,a,d\in Q$ with $ha=d^n$ for some
$n\geq 1$ and $a$ is elliptic, then $|h|=n|d|$.}

\medskip \noindent
Write $C:=\Min(h)\cap\Min(a)$. Then $C$ is nonempty by Claim 4.
Since $d^n$ stabilizes $C$, Claim 5 implies that it is semi-simple
in $C$ with translation length $|d^n|_C=|d^n|$. Thus,
$|d^n|_C=|d^n|=n|d|$ by Claim 6. Note also that $ha$ is semi-simple
in $C$ with translation length $|ha|_C=|h|$. Therefore,
$|h|=|ha|_C=|d^n|_C=n|d|$, as desired.

\medskip \noindent
\underline{Claim 9}: \emph{Let $h\in Q$ be hyperbolic. Suppose that
$ha_i=d_i^{n_i}$ for all $i\geq 1$, where $a_i,d_i\in Q$, each $a_i$
is elliptic and where $n_i\geq 1$. Then the set $\{n_i \ | \ i\geq
1\}$ is bounded.}

\medskip \noindent
Indeed, by Claim 8, the sequence $(d_i)$ of elements of $Q$ is such
that $|d_i|=|h|/n_i>0$. The claim now follows from the second part
of Proposition~\ref{thm Bridson}.

\medskip \noindent
\underline{Claim 10}: \emph{Let $p\in\PP$. Then all elements of
$Q_p$ are elliptic.}

\medskip \noindent
Suppose for a contradiction there exists an $h\in Q_p$ which is not
elliptic, and is thus hyperbolic by Proposition~\ref{thm Bridson}.
Let $q\in\PP\setminus\{p\}$. Since $P_p=\pi(Q_p)$ is $q$-divisible,
there exists an $h_q\in Q$ such that $h_q^qQ^0=hQ^0$. Let $a\in Q^0$
such that $ha=h_q^q$. By Proposition~\ref{prop compact iff
divisible}, since $Q^0$ is compact and connected, it is divisible,
and so $a$ is elliptic by Claim 7. Since the set of natural prime
numbers distinct from $p$ is unbounded, the desired contradiction
now comes from Claim 9.

\medskip
Let now $gQ^0=(g_p)_{p\in\PP}$ be the decomposition of $\pi(g)$ in
$P=\prod_{p\in\PP}{P_p}$ (that is, each $g_p\in P_p$). Let
$q\in\PP$, and choose an $a_q\in Q_p$ such that $\pi(a_q)=g_q\inv$.
Then $\pi(ga_q)$ has no component in the $q$-Sylow of $P$, and is
therefore $q$-divisible in $P$. Hence, there exist an $h_q\in Q$ and
an $a\in Q^0$ such that $ga_qa=h_q^q$. By Claim 10, we know that
$a_q$ is elliptic. But so is $a$, and hence the product $a':=a_qa$
is also elliptic by Claim 4. Since $q$ is an arbitrary prime, Claim
9 again yields the desired contradiction.
\end{proof}

\subsection{The Davis realization of a building}
We recall from~\cite{Davis} that any building $\Delta$ admits a
metric realization, denoted by $|\Delta|_{\CAT}$, which is a locally Euclidean $\CAT$
cell complex with finitely many types of cells. Moreover any group
of type-preserving automorphisms of $\Delta$ acts in a canonical way
by cellular isometries on $|\Delta|_{\CAT}$. Finally, the cell
supporting any point of $|\Delta|_{\CAT}$ determines a unique
spherical residue of $\Delta$. In particular, an automorphism of
$\Delta$ which fixes a point in $|\Delta|_{\CAT}$ must stabilize the
corresponding spherical residue in $\Delta$.

Here is a reformulation of Theorem~\ref{1}.

\begin{thm}
Let $G$ be a compact group acting strongly transitively by type-preserving automorphisms on a thick building $\Delta$. Then, $\Delta$ is spherical.
\end{thm}
\begin{proof}
Let $(W,S)$ be the Coxeter system associated to $\Delta$, and let $\Sigma$ be the fundamental apartment of $\Delta$. Then, the action of the stabilizer in $G$ of $\Sigma$ can be identified with the action of $W$ on this apartment (\cite[2.8]{Weiss}).

\medskip \noindent
\underline{Claim 1}: \emph{$|\Sigma|_{\CAT}$ is a closed convex
subset of $|\Delta|_{\CAT}$.}

\medskip \noindent
A basic fact about buildings is the existence, for each pair
$(\Sigma,C)$ consisting of an apartment $\Sigma$ and of a chamber
$C\in\Sigma$, of a \emph{retraction of $\Delta$ onto $\Sigma$
centered at $C$}, that is, of a simplicial map
$\rho=\rho_{\Sigma,C}\co\Delta\to \Sigma$ preserving minimal
galleries from $C$ and such that $\rho_{|\Sigma}=\id_{|\Sigma}$. The
induced mapping $\overline{\rho}\co |\Delta|_{\CAT}\to
|\Sigma|_{\CAT}$ then maps every geodesic segment of
$|\Delta|_{\CAT}$ onto a piecewise geodesic segment of
$|\Sigma|_{\CAT}$ of same length. In particular, the mapping
$\overline{\rho}$ is distance decreasing (see~\cite[Lemme
11.2]{Davis}). Hence, if $x$ and $y$ are two points in
$|\Sigma|_{\CAT}$, then the geodesic segment from $x$ to $y$ is
entirely contained in $|\Sigma|_{\CAT}$ since its image by
$\overline{\rho}$ is also a geodesic from $x$ to $y$. This proves
that $|\Sigma|_{\CAT}$ is convex. To see it is closed, it suffices
to note that it is complete as a metric space since it is precisely
the Davis realization of the building $\Sigma$.

\medskip \noindent
\underline{Claim 2}: \emph{If $g\in G$ is elliptic in
$X=|\Delta|_{\CAT}$ and stabilizes $|\Sigma|_{\CAT}$, then $g$ is
also elliptic in $|\Sigma|_{\CAT}$.}

\medskip \noindent
This follows from Claim 5 in the proof of Theorem~\ref{thm compact
elliptic}.

Theorem~\ref{thm compact elliptic} now implies that the
induced action of $W$ on $|\Sigma|_{\CAT}$ is elliptic, that is,
every $w\in W$ is elliptic. Notice that the $W$-action on
$|\Sigma|_{\CAT}$ is proper, since by construction, it is cellular and the stabilizer of every point is a spherical (in particular finite) parabolic subgroup of $W$. Recalling now that every infinite finitely
generated Coxeter group contains elements of infinite order (in
fact, so do all finitely generated infinite linear groups by a
classical result of Schur~\cite{Schur}; in the special case of
Coxeter groups, a direct argument may be found in~\cite[Proposition
2.74]{ABrown}), we deduce that $W$ is finite. In other words
$\Delta$ is spherical.
\end{proof}

\section{Proof of Theorem 2}
\subsection{Heuristic sketch}
Let $G$ be a compact group possessing a split spherical BN-pair, and
let $\Delta$ be the associated building. We first establish Theorem
2 when $G$ acts continuously on $\Delta$. In that case,
$2$-transitive actions (which are closely related to strongly
transitive actions) of $G$ on subspaces $X$ of $\Delta$ are easily
seen to be possible only for finite $X$. The second step is then to
show that the action of $G$ on $\Delta$ has to be continuous. This
uses the fact that buildings arising from split spherical BN-pairs
are Moufang (see Proposition~\ref{prop V transitive} below).

\subsection{Continuous actions on buildings}
Recall that a topological space $X$ is said to satisfy the {\bf
$T_1$} separation axiom when all its singletons are closed. The
following is probably well-known.

\begin{lemma}\label{thm 2-transitive continuous action}
Let $G$ be a compact group. If $G$ admits a continuous
$2$-transitive action on a $T_1$ topological space $X$, then $X$ is
finite.
\end{lemma}
\begin{proof}
Define $Y:=\{(x,y)\in X\times X \ | \ x\neq y\}\subset X\times X$, and fix $x,y\in X$ with $x\neq y$. Since the mapping $\alpha_x\co G\to X: g\mapsto g\cdot x$ is continuous, so is $\alpha_x\times \alpha_y\co G\to X\times X: g\mapsto (g\cdot x, g\cdot y)$. By $2$-transitivity, we get $Y= (\alpha_x\times \alpha_y)(G)$, and so $Y$ is compact.

Note also that the mapping $f\co X\times X\to X\times X: (a,b)\mapsto (x,b)$ is continuous. Setting $Z:=X\setminus\{x\}$, we then get $Z\times\{x\}= f^{-1}(\{(x,x)\})\cap Y$, so that $Z\times\{x\}$ is closed in $Y$, and hence compact. It follows that $Z$ is compact, being the image of $Z\times\{x\}$ by the projection on the first factor $X\times X\to X$, which is of course continuous.

In particular, $Z$ is closed, and hence $\{x\}$ is open. It follows that $X$ is discrete, and therefore finite since $X=\alpha_x(G)$ is compact.
\end{proof}

Let $\Delta$ be a building of type $(W,S)$, and denote by
$\Ch\Delta$ the set of its chambers. Consider the chamber system
$\Gamma$ of $\Delta$, which is the labelled graph with vertex set
$\Ch\Delta$ and with an edge labelled by $s\in S$ for each pair of
$s$-adjacent chambers of $\Delta$ (see~\cite[Ch.I Appendix
D]{Brown}). Let $J\subset S$. A {\bf $J$-gallery} in $\Gamma$
between two chambers $x$ and $y$ of $\Delta$ is a sequence
$(x=x_0,x_1,\dots,x_l=y)$ of chambers of $\Delta$ such that for each
$i\in\{1,\dots,l\}$, there exists an $s\in J$ such that $x_{i-1}$ is
$s$-adjacent to $x_i$. The natural number $l$ is called the {\bf
length} of the gallery. A {\bf minimal} gallery is a gallery of
minimal length. The {\bf distance} in $\Delta$ between two chambers
$x,y\in\Ch\Delta$ is the length of a minimal gallery joining $x$ to
$y$. The {\bf diameter} of $\Gamma$ is the supremum (in
$\NN\cup\{\infty\}$) of the distances between its vertices.

Let $J\subset S$. The {\bf $J$-residue} $R=R_J(x)$ of some chamber $x\in\Ch\Delta$ is the set of chambers of $\Delta$ which are connected to $x$ by a $J$-gallery. When $J$ has cardinality $1$, we call $R$ a {\bf panel}.

In this paper, we will say that a group $G$ acts {\bf continuously} on $\Delta$ if the stabilizers of the residues of $\Delta$ are closed in $G$. Note that we can of course restrict our attention to the maximal proper residues, the others being obtained as intersections of those.

\begin{lemma}\label{thm continuous}
Let $G$ be a compact group acting continuously and strongly
transitively by type-preserving automorphisms on a spherical thick building $\Delta$. Then $\Delta$
is finite.
\end{lemma}
\begin{proof}
The stabilizer $H$ in $G$ of a panel $P$ of $\Delta$ is a closed and thus compact subgroup of $G$.

\medskip \noindent
\underline{Claim 1}: \emph{$H$ acts $2$-transitively on $P$.}

\medskip \noindent
Indeed, let $C$ be a chamber of $P$ and let $B:=\Stab_G(C)\subset
H$. We first show that $B$, and thus also $H$, is transitive on the
set $\CCC=P\setminus \{C\}$. Let $C_1, C_2\in\CCC$ and let
$\Sigma_1$ (respectively, $\Sigma_2$) be an apartment containing $C$
and $C_1$ (respectively, $C$ and $C_2$). By strong transitivity, $B$
is transitive on the set of apartments containing $C$, and so there
exists a $b\in B$ such that $b\Sigma_1=\Sigma_2$. Hence $bC_1=C_2$.
It now remains to show that $H$ is transitive on $P$. But if
$C_1,C_2\in P$, then since $\Delta$ is thick, we may choose a
chamber $C$ in $P$ different from $C_1,C_2$. The stabilizer $B'$ of
$C$ in $G$ is then contained in $H$ and is transitive on $P\setminus
\{C\}$ by the previous argument.

Now, identifying $\Delta$ with $\Delta(G,B)$, so that $H=B\cup BsB$
for some generator $s$ of the corresponding Weyl group, we get a
$2$-transitive, continuous action by left translation of the compact
group $H$ on the topological space $H/B$. Moreover, this space is
$T_1$ since $B$ is closed in $G$ by hypothesis. Lemma~\ref{thm
2-transitive continuous action} then implies that $P$ is finite. In
other words, as $P$ was arbitrary, the building $\Delta$ is
\emph{locally finite}, that is, every panel is finite. The following
observation now allows us to conclude:

\medskip \noindent
\underline{Claim 2}: \emph{Every locally finite spherical building
is finite.}

\medskip \noindent
Indeed, let $\Gamma=\Ch\Delta$ be the graph whose vertices are the
chambers of $\Delta$, and such that two chambers of $\Delta$ are
adjacent if they share a common panel. Since $\Delta$ is locally
finite, so is $\Gamma$. Hence, fixing a vertex $x\in\Gamma$, each
ball in $\Gamma$ centered at $x$ with radius $n$ ($n\in\NN$)
possesses a finite number of vertices. Moreover, as $\Delta$ is
spherical, the diameter of $\Delta$ is finite (\cite[Ch.IV,
3]{Brown}), and hence the diameter of $\Gamma$ is also finite. Thus
$\Gamma$ is contained in such a ball, and is therefore finite.
\end{proof}

\subsection{Moufang buildings}

Let $\Delta=\Delta(G,B)$ be the building associated to a split
spherical BN-pair $(B=T\ltimes U,N)$ of type $(W,S)$. It is well-known (see the main result of \cite{Moufang}) that the existence of a splitting for the above BN-pair is equivalent to the fact that the building $\Delta$ enjoys the Moufang property, as defined in \cite[Chapter~11]{Weiss}.

Two chambers
$x,y\in\Ch\Delta$ are called {\bf opposite} if they are at maximal
distance in the chamber system of $\Delta$. Similarly, one can
define \emph{opposite residues} (see for instance
\cite[5.7]{ABrown}). The set of chambers (respectively, residues) of
$\Delta$ which are opposite to a given chamber $C$ (respectively,
residue $R$) will be denoted by $C^{\op}$ (respectively, $R^{\op}$).

\begin{prop}\label{prop V transitive}
Let $P=BW_JB$ be a proper standard parabolic subgroup of $\Delta=\Delta(G,B)$ for some proper subset $J$ of $S$, let $C$ be the fundamental chamber (i.e. the unique chamber fixed by $B$) and let $R$ be the unique $J$-residue containing $C$. Define the subgroup $V:=\bigcap_{p\in P}{pUp\inv}$ of $G$. Then $V$ acts simply transitively on $R^{\op}$.
\end{prop}
\begin{proof}
Let $\Sigma$ be an apartment containing $C$. By \cite[9.11]{Weiss}, there exists a
minimal galery $\gamma_{R'}$, one for each residue $R'\in R^{\op}$, beginning at $C$ and ending at a chamber $C'$ in
$R'$ such that the type of $\gamma_{R'}$ is independent of the choice of $R'$ and $C = \proj_RC'$. Let $R'  \in R^{\op}$ be the unique residue of $\Sigma$ opposite $R$ and let $C'$ be the last chamber of $\gamma_{R'}$. Let also $\alpha$ be a root of $\Sigma$ containing $C$ but not $C'$. By \cite[8.21]{Weiss}, $R\cap \Sigma\subset\alpha$. By \cite[9.7]{Weiss}, therefore, $R$ is fixed pointwise by the root group
$U_{\alpha}$. Since $P$ maps $R$ to itself, we have $C\in R\subset\alpha^p$ and hence $p\inv U_{\alpha}p\subset U$ for all $p\in P$ by the definition of root subgroups (see \cite[11.1]{Weiss}) and the fact that the `radical' $U$ does not depend on the choice of the apartment $\Sigma$ (see  \cite[Proposition~11.11(iii)]{Weiss}). Thus $U_{\alpha}\subset V$. Now, as in \cite[7.67]{ABrown}, one shows that the subgroup of $V$ generated by all $U_\alpha$'s of the latter form acts transitively on the set $\{\gamma_{R''} \; | \; R'' \in R^{\op}\}$, and hence also  on $R^{\op}$.

Suppose $h\in V$ maps $R'\in R^{\op}$
to itself. Then $h$ acts trivially on $R$. Since the restriction of $\proj_{R'}$ to $R$ is a bijection from $R$ to
$R'$ (by \cite[9.11]{Weiss} again), it follows that $h$ acts trivially on $R'$. By \cite[9.8]{Weiss}, therefore, $h$ fixes
two opposite chambers of $\Sigma$ and hence $h$ fixes $\Sigma$. By \cite[9.7]{Weiss} again, we conclude that $h = 1$.
\end{proof}

In particular, we have the following (compare~\cite[Ch.IV,
5]{Brown}).

\begin{lemma}\label{lemma U transitive}
Let $C$ be the fundamental chamber of $\Delta$. Then $U$ acts simply
transitively on $C^{\op}$. Equivalently, $U$ acts simply
transitively on the set of apartments containing $C$.
\end{lemma}

\begin{lemma}\label{lemma at least two opposite residues}
Let $P=BW_JB$ be a proper standard parabolic subgroup of $\Delta=\Delta(G,B)$ for some proper subset $J$ of $S$, let $C$ be the fundamental chamber and let $R$ be the unique $J$-residue containing $C$. Then there exist two chambers in $C^{\op}$
which are opposite to one another. In particular, $|R^{\op}|\geq 2$.
\end{lemma}
\begin{proof}
The first assertion holds by \cite[Proposition 4.104]{ABrown} and the second follows since no proper residue contains two opposite chambers.
\end{proof}

We are now ready to complete the proof of Theorem~\ref{2}.

\begin{thm}
Let $G$ be a compact topological group possessing a spherical split BN-pair $(B=T\ltimes U,N)$. Then the associated building is finite.
\end{thm}
\begin{proof}
Let $\Delta=\Delta(G,B)$ be the building associated to $(B,N)$, and
let $(W,S)$ be the corresponding Coxeter system.

We start with some basic observations in the case $(W,S)$ is not
irreducible. Suppose thus that $S$ decomposes as $S=S_1\amalg S_2$
with $s_1s_2=s_2s_1$ for all $s_1\in S_1$ and $s_2\in S_2$. Then $W$ splits as a direct product $W \cong W_1 \times W_2$, where $W_i = \la S_i\ra$, and the building $\Delta$ decomposes canonically as a product $\Delta = \Delta_1 \times \Delta_2$ of buildings of type $(W_1, S_1)$ and $(W_2, S_2)$ respectively (see \cite[Proposition~7.33]{Weiss}). 

In particular, we obtain induced actions of $G$ on both $\Delta_1$ and $\Delta_2$, which are obviously strongly transitive. The corresponding BN-pairs for $G$ may be described as follows. Since
each $s\in S$ can be written as a coset $nT\in N/T=W$, we may
choose, for $i=1,2$, a set $\overline{N}_i$ of representatives in
$N$ for the elements of $S_i$. For each $i=1,2$, consider now the
subgroup $N_i$ of $N$ generated by $\overline{N}_i$ and $T$, and set
$B_i:=\la B\cup N_{3-i}\ra = B N_{3-i} B \leq G$. Then $(B_i,N_i)$ is a
spherical BN-pair for $G$, and the associated building is nothing but $\Delta_i = \Delta(G,B_i)$.

We claim that the BN-pair  $(B_i, N_i)$ is split. This follows readily
from the aforementioned equivalence between splittings of BN-pairs
and the Moufang property for the associated buildings. More
precisely, consider the group $U_i = \bigcap_{g \in B_{i}} g U
g^{-1}$ which is the kernel of the $U$-action on $\Delta_{3-i}$.
Then $U_i$ acts sharply transitively on the chambers of $\Delta_i$
which are opposite the standard chamber $C$, which by definition is
the unique chamber fixed by $B_i$. Therefore we have $B_i \cong T_i
\ltimes U_i$, where $T_i = \bigcap_{w \in W_i} w B_i w^{-1}$, and
$U_i$ induces a splitting of the BN-pair $(B_i, N_i)$ as claimed.

This shows that the given split BN-pair for $G$ yields various split
BN-pairs for $G$ corresponding to the various irreducible components
of $\Delta$. Since $\Ch \Delta$ is naturally in one-to-one
correspondence with the Cartesian product $\Ch \Delta_1 \times
\cdots \times \Ch \Delta_n$ of the chamber sets of the various
irreducible components of $\Delta$, the desired finiteness result
readily follows provided we establish it for each irreducible
BN-pair $(B_i, N_i)$ as above. In other words, there is no loss of
generality in assuming that the building $\Delta$ is irreducible. We
adopt henceforth this additional assumption.

\medskip

Let now $\PPP$ denote the set of maximal proper standard parabolic
subgroups of $G$. Pick any $P\in\PPP$. Thus $P$ is of the form $P=BW_JB$
for some maximal subset $J\subsetneq S$, where $W_J=\la J\ra$. In
particular, $P$ is a maximal subgroup of $G$ (see~\cite[Lemma
6.43(1)]{ABrown}). Define the normal subgroup
$$V:=\bigcap_{p\in P}{pU p\inv}\trianglelefteq P$$
of $P$. As $V$ is contained in $U$, it is also nilpotent. Moreover,
$V$ acts faithfully on $\Delta$. Indeed, the kernel
$\ker(G\curvearrowright \Delta)$ of the action of $G$ on $\Delta$ is
obviously contained in the stabilizer of the chambers of the
fundamental apartment $\Sigma$, that is, in $\bigcap_{w\in
W}{wBw\inv}=T$, and so
$$V \cap \ker(G\curvearrowright
\Delta)\subseteq U\cap T=\{1\}.$$

Now, since $V$ is normal in $P$, we have $P\subseteq \norma_G(V)$.
Moreover, as the conjugation automorphism $\kappa_g\co G\to G:
x\mapsto gxg\inv$ is continuous, we get
$\norma_G(\overline{V})\supseteq \norma_G(V)$ and so
$\norma_G(\overline{V})\supseteq P$. Hence, by maximality of $P$, we
obtain that either $\norma_G(\overline{V})=P$ or
$\norma_G(\overline{V})=G$.

%Repeating this argument for every $P\in\PPP$, we are left with two cases to consider.

\medskip \noindent \underline{Claim}: \emph{$\norma_G(\overline{V})= P$ for
all $P\in\PPP$.}

\medskip \noindent%
Assume for a contradiction that $\norma_G(\overline{V})= G$ for some
$P\in\PPP$. In other words, $\overline{V}\triangleleft G$. In
particular, the center $\centra(\overline{V})\subseteq \overline{V}$
of $\overline{V}$ is also a normal subgroup of $G$. Moreover, $V$ is
nontrivial since, by Proposition~\ref{prop V transitive}, it acts
transitively on $R^{\op}$ and since $|R^{\op}|\geq 2$ by
Lemma~\ref{lemma at least two opposite residues}. As $V$ is
nilpotent, this implies that $\centra(V)$ is also nontrivial.

Now, using again the continuity of the conjugation automorphism
$\kappa_h$ (for $h\in G$), we see that $\centra(V)=\centra_G(V)\cap
V$ is contained in
$\centra(\overline{V})=\centra_G(\overline{V})\cap \overline{V}$.
Moreover, as $V$ acts faithfully on $\Delta$, so does $\centra(V)$.
This implies in particular that $\centra(V)$, and thus also
$\centra(\overline{V})$, act nontrivially on $\Delta$.

Tits' transitivity Lemma (see \cite[Lemma 6.61]{Brown}) then guarantees
that the group $\centra(\overline{V})$ is transitive on the chambers
of $\Delta$. In fact, this action is even simply transitive. Indeed,
the stabilizers in $\centra(\overline{V})$ of the chambers of
$\Delta$ are all conjugate by transitivity. They are thus all equal
since $\centra(\overline{V})$ is Abelian, and are therefore
contained in the kernel $\ker(G\curvearrowright \Delta)$ of the
action of $G$ on $\Delta$. Since $\centra(V)\subseteq
\centra(\overline{V})$, this implies that the action of $\centra(V)$
on $\Ch\Delta$ is free. But since $\centra(V)\subseteq V\subseteq
U\subseteq B$, and as $B$ stabilizes the fundamental chamber, it
follows that $\centra(V)$ acts trivially on $\Delta$. This
contradiction establishes the Claim.

\medskip

Since the normalizer of a closed subgroup is closed, we deduce from
the Claim that every $P\in\PPP$ is closed.  But this means that $G$
acts continuously on $\Delta$, and so Lemma~\ref{thm continuous}
ensures that $\Delta$ is finite, as desired.
\end{proof}

\section{Proof of Theorem 4}

Let $k$ be a perfect field and let $K=\overline{k}$ be its algebraic closure. In what follows, we identify an algebraic $k$-group $G$ with its group of $K$-rational points.

The main tool for the proof of Theorem 4 is the following
characterization of anisotropy, due to Borel and Tits
(see~\cite{BTunipotent}).

\begin{prop}\label{thm BT unipotent}
Let $G$ be a reductive algebraic $k$-group and let $U$ be a
unipotent $k$-subgroup of $G$. If $k$ is perfect, then there exists
a parabolic $k$-subgroup $P$ of $G$ whose unipotent radical $R_u(P)$
contains $U$.
\end{prop}

In particular, if $G$ is anisotropic over $k$, then $U$ must be
trivial.

\begin{proof}[Proof of Theorem 4]
Suppose for a contradiction that the split spherical BN-pair $(B,N)$
for the reductive $k$-group $G$ is such that $B$ has infinite index
in $G(k)$. Let $\Delta=\Delta(G(k),B)$ be the associated building,
and let $W$ be the corresponding (finite) Weyl group. Also, denote
by $\overline{B}$ the Zariski closure of $B$ in $G$.

The Bruhat decomposition for $G$ yields $G=\coprod_{w\in W}{BwB}$.
Since $G(k)$ is Zariski dense in $G$ by \cite[18.3]{Borel}, we have
$$G=\overline{G(k)} = \overline{\coprod_{w\in W} BwB} \subseteq \coprod_{w\in W}{\overline{BwB}}.$$
As $G$ is connected, it cannot be written as a finite union of
closed subsets in a nontrivial way. Therefore, we deduce that $BwB$
is dense in $G$ for some $w \in W$. In particular, so is $\overline
B w \overline B$.

Let now $A:=(\overline{B})^0$ be the identity component of
$\overline B$. Since $A$  has finite index in $\overline B$, it
follows that $\overline Bw \overline B$ is a finite union of double
cosets modulo $A$. As before, this implies that some double coset of
the form $AzA$ is dense in $G$.

\medskip \noindent \underline{Claim}: \emph{$\overline{B}\neq G$.}

\medskip \noindent
Indeed, let $U$ be the nilpotent normal subgroup of $B$ arising from
the splitting of the BN-pair, and suppose for a contradiction that
$B$ is dense in $G$. Then the Zariski closure $\overline{U}$ of $U$
in $G$ is a nilpotent normal subgroup of $\overline{B}=G$
(\cite[2.1]{Borel}). Its identity component $\overline{U}^0$ is thus
contained in the radical of $G$, which coincides with the connected
center $\centra(G)^0$ (\cite[11.21]{Borel}). Hence, since
$\overline{U}^0$ has finite index in $\overline{U}$, we get
$$[U:U\cap \centra(G)]\leq [U:U\cap \overline{U}^0]=
[U\overline{U}^0:\overline{U}^0]\leq
[\overline{U}:\overline{U}^0]<\infty.$$ Now, if $u\in U\cap
\centra(G)$, then $u$ acts trivially on $\Delta$ since for any
chamber $gB$, we have $ugB=guB=gB$. As $U$ acts simply transitively
on $C^{\op}$ by Lemma~\ref{lemma U transitive}, where $C=1_GB$ is
the fundamental chamber of $\Delta$, this implies that $u=1$:
otherwise, $\Delta$ would contain only one apartment, so that
$[G(k):B]<\infty$, a contradiction. So $U\cap \centra(G)=\{1\}$ and
therefore $U$ is finite. Using again the sharp transitivity of $U$
on $C^{\op}$, we deduce that $\Delta$ is the reunion of finitely
many apartments, hence is finite,  contradicting once more our
initial hypothesis. The claim stands proven.

\medskip
In particular $A$ is a proper closed subgroup of $G$ such that $Az
A$ is dense in $G$ for some $z \in G$.  The main result
of~\cite{Brundan} now implies that $A$ is not reductive, \emph{i.e.}
the unipotent radical $R_u(A)$ is nontrivial. Moreover, since $B$ is
contained in $G(k)$ and is dense in $\overline{B}$, we know that
$\overline{B}$ is defined on $k$ (\cite[AG.14.4]{Borel}). Hence, $A$
is also $k$-defined (\cite[1.2]{Borel}), and so is $R_u(A)$ since
$k$ is perfect (\cite[12.1.7(d)]{Springer}). Thus $R_u(A)$ is a nontrivial unipotent $k$-subgroup of $G$. As we observed following Proposition \ref{thm BT unipotent}, this contradicts the assumption that $G$ is anisotropic over $k$.
\end{proof}

%\bibliographystyle{plain}
%\bibliography{AnisotropicBN}

\begin{thebibliography}{10}

\bibitem{AB07}
Peter Abramenko and Kenneth~S. Brown.
\newblock Transitivity properties for group actions on buildings.
\newblock {\em J. Group Theory}, 10(3):267--277, 2007.

\bibitem{ABrown}
Peter Abramenko and Kenneth~S. Brown.
\newblock {\em Buildings}, volume 248 of {\em Graduate Texts in Mathematics}.
\newblock Springer, New York, 2008.
\newblock Theory and applications.

\bibitem{BTunipotent}
Armand~Borel and Jacques~Tits.
\newblock \'{E}l\'ements unipotents et sous-groupes paraboliques de groupes
  r\'eductifs. {I}.
\newblock {\em Invent. Math.}, 12:95--104, 1971.

\bibitem{Borel}
Armand Borel.
\newblock {\em Linear algebraic groups}, volume 126 of {\em Graduate Texts in
  Mathematics}.
\newblock Springer-Verlag, New York, second edition, 1991.

\bibitem{BT}
Armand Borel and Jacques Tits.
\newblock Groupes r\'eductifs.
\newblock {\em Inst. Hautes \'Etudes Sci. Publ. Math.}, (27):55--150, 1965.

\bibitem{BridsonHaefliger}
Martin~R. Bridson and André~Haefliger.
\newblock {\em Metric spaces of non-positive curvature}, volume 319 of {\em
  Grundlehren der Mathematischen Wissenschaften}.
\newblock Springer-Verlag, Berlin, 1999.

\bibitem{Bridson}
Martin~R. Bridson.
\newblock On the semisimplicity of polyhedral isometries.
\newblock {\em Proc. Amer. Math. Soc.}, 127(7):2143--2146, 1999.

\bibitem{Brown}
Kenneth~S. Brown.
\newblock {\em Buildings}.
\newblock Springer Monographs in Mathematics. Springer-Verlag, New York, 1998.
\newblock Reprint of the 1989 original.

\bibitem{Brundan}
Jonathan Brundan.
\newblock Double coset density in reductive algebraic groups.
\newblock {\em J. Algebra}, 177(3):755--767, 1995.

\bibitem{Davis}
Michael~W. Davis.
\newblock Buildings are {${\rm CAT}(0)$}.
\newblock In {\em Geometry and cohomology in group theory ({D}urham, 1994)},
  volume 252 of {\em London Math. Soc. Lecture Note Ser.}, pages 108--123.
  Cambridge Univ. Press, Cambridge, 1998.

\bibitem{Moufang}
Tom~De~Medts, Fabienne~Haot, Katrin~Tent, and Hendrik~Van~Maldeghem.
\newblock Split {$BN$}-pairs of rank at least 2 and the uniqueness of
  splittings.
\newblock {\em J. Group Theory}, 8(1):1--10, 2005.

\bibitem{connexedivisible}
Jan Mycielski.
\newblock Some properties of connected compact groups.
\newblock {\em Colloq. Math.}, 5:162--166, 1958.

\bibitem{Prasad}
Gopal Prasad.
\newblock Elementary proof of a theorem of {B}ruhat-{T}its-{R}ousseau and of a
  theorem of {T}its.
\newblock {\em Bull. Soc. Math. France}, 110(2):197--202, 1982.

\bibitem{Rapinchukprasad}
Andrei~S. Rapinchuk.
\newblock Algebraic and {A}bstract {S}imple {G}roups: {O}ld ans {N}ew.
\newblock {\em Prepublication}, pages 55--150, 2003.

\bibitem{Rapinchuk}
Andrei~S. Rapinchuk, Yoav Segev, and Gary~M. Seitz.
\newblock Finite quotients of the multiplicative group of a finite dimensional
  division algebra are solvable.
\newblock {\em J. Amer. Math. Soc.}, 15(4):929--978 (electronic), 2002.

\bibitem{Procyclic}
Luis Ribes and Pavel Zalesskii.
\newblock {\em Profinite groups}, volume~40 of {\em Ergebnisse der Mathematik
  und ihrer Grenzgebiete. 3. Folge. A Series of Modern Surveys in Mathematics}.
\newblock Springer-Verlag, Berlin, 2000.

\bibitem{Schur}
Issai Schur.
\newblock {\"Uber Gruppen periodischer linearer Substitutionen.}
\newblock {\em Berl. Ber.}, 1911:619--627, 1911.

\bibitem{Springer}
Tonny~A. Springer.
\newblock {\em Linear algebraic groups}, volume~9 of {\em Progress in
  Mathematics}.
\newblock Birkh\"auser Boston Inc., Boston, MA, second edition, 1998.

\bibitem{Tent}
Katrin Tent.
\newblock A short proof that root groups are nilpotent.
\newblock {\em J. Algebra}, 277(2):765--768, 2004.

\bibitem{Ti64}
Jacques~Tits.
\newblock Algebraic and abstract simple groups.
\newblock {\em Ann. of Math. (2)}, 80:313--329, 1964.

\bibitem{Weiss}
Richard M. Weiss.
\newblock {\em The structure of spherical buildings.}
\newblock Princeton University Press, Princeton, NJ, 2003.


\end{thebibliography}
\end{document}